\documentclass[11pt]{amsart}
\usepackage{amsmath,amsthm,amssymb}
\usepackage{hyperref,graphics,color}

\newtheorem{theorem}{Theorem}[section]

\newtheorem{corollary}[theorem]{Corollary}
\newtheorem{remark}[theorem]{Remark}

\title [Some results for the Apostol-Genocchi polynomials]{Some results for the Apostol-Genocchi polynomials of higher order}

\author{$^{1*}$Hassan Jolany and $^{\dag}$Hesam Sharifi} 

\begin{document}
\maketitle
\begin{center}
\footnote{corresponding Author}*School of Mathematics, Statistics
and Computer Science, University of Tehran, Iran. E-mail:
hassan.jolany@khayam.ut.ac.ir
\end{center}

\begin{center}
$^{\dag}$Department of Mathematics, Faculty of Science,
University of Shahed, Tehran, Iran. E-mail: hsharifi@shahed.ac.ir
\end{center}




\date{}
\begin{abstract}
The present paper deals with multiplication formulas for the
Apostol-Genocchi polynomials of higher order and deduces some
explicit recursive formulas. Some earlier results of Carlitz and
Howard in terms of Genocchi numbers can be deduced. We introduce
the 2-variable Apostol-Genocchi polynomials and then we consider
the multiplication theorem for 2-variable Genocchi polynomials.
Also we introduce generalized Apostol-Genocchi polynomials with
$a,b,c$ parameters and we obtain several identities on
generalized Apostol-Genocchi polynomials with $a,b,c$ parameters .
\\

\end{abstract}
\section{Preliminaries and motivation}

The classical Genocchi numbers are defined in a number of ways.
The way in which it is defined is often determined by which sorts
of applications they are intended to be used for. The Genocchi
numbers have wide-ranging applications from number theory and
Combinatorics to numerical analysis and other fields of applied
mathematics. There exist two important definitions of the
Genocchi numbers: the generating function definition, which is
the most commonly used definition, and a Pascal-type triangle
definition, first given by Philip Ludwig von Seidel, and
discussed in [29]. As such, it makes it very appealing for use in
combinatorial applications. The idea behind this definition, as
in Pascal's triangle, is to utilize a recursive relationship
given some initial conditions to generate the Genocchi numbers.
The combinatorics of the Genocchi numbers were developed by Dumont
in [4] and various co-authors in the 70s and 80s. Dumont and
Foata introduced in 1976 a three-variable symmetric refinement of
Genocchi numbers, which satisfies a simple recurrence relation. A
six-variable generalization with many similar properties was
later considered by Dumont.  In [30] Jang et al. defined a new
generalization of Genocchi numbers, poly Genocchi numbers. Kim in
[10] gave a new concept for the q-extension of Genocchi numbers
and gave some relations between q-Genocchi polynomials and q-Euler
numbers. In [31], Simsek et al. investigated the q-Genocchi zeta
function and L-function by using generating functions and Mellin
transformation. Genocchi numbers are known to count a large
variety of combinatorial objects, among which numerous sets of
permutations. One of the applications of Genocchi numbers that was
investigated by Jeff Remmel in [32] is counting the number of
up-down ascent sequences. Another application of Genocchi numbers
is in Graph Theory. For instance, Boolean numbers of the
associated Ferrers Graphs are the Genocchi numbers of the second
kind [33]. A third application of Genocchi numbers is in Automata
Theory. One of the generalizations of Genocchi numbers that was
first proposed by Han in [34] proves useful in enumerating the
class of deterministic finite automata (DFA) that accept a finite
language and in enumerating a generalization of permutations
counted by Dumont. Recently S. Herrmann in [6], presented a
relation between the $f$ -vector of the boundary and the interior
of a simplicial ball directly in terms of the $f$-vectors. The
most interesting point about this equation is the occurrence of
the Genocchi numbers $G_{2n}$. In the last decade, a surprising
number of papers appeared proposing new generalizations of the
classical Genocchi polynomials to real and complex variables or
treating other topics related to Genocchi polynomials. Qiu-Ming
Luo in [19] introduced new generalizations of Genocchi
polynomials, he defined the Apostol-Genocchi polynomials of
higher order and q-Apostol-Genocchi polynomials and he obtained a
relationship between Apostol-Genocchi polynomials of higher order
and Goyal-Laddha-Hurwitz-Lerch Zeta function. Next Qiu-Ming Luo
and H.M. Srivastava in [35] by Apostol-Genocchi polynomials of
higher order derived various explicit series representations in
terms of the Gaussian hypergeometric function and the Hurwitz (or
generalized) zeta function which yields a deeper insight into the
effectiveness of this type of generalization. Also it is clear
that Apostol-Genocchi polynomials of higher order are in a class
of orthogonal polynomials and we know that most such special
functions that are orthogonal are satisfied in multiplication
theorem, so in this present paper we show this property is true
for Apostol-Genocchi polynomials of higher order.

The study of Genocchi numbers and their combinatorial relations
has received much attention
\cite{Ca,Du,He,Ki1,Li2,Lu5,Ri1,Ri2,Ry2,Sim2,Zh,Zq}. In this paper
we consider some combinatorial relationships of the
Apostol-Genocchi numbers of higher order.

 The unsigned Genocchi numbers $\{G_{2n}
\}_{n\geqslant 1}$ can be defined through their generating
function:

\begin{equation*}
\sum_{n=1}^{\infty} G_{2n}
\frac{x^{2n}}{(2n)!}=x.\tan\Big(\frac{x}{2}\Big)
\end{equation*}
and also
$$\sum_{n\geqslant 1}(-1)^{n}G_{2n}\frac{t^{2n}}{(2n)!}=-t\tanh \Big(\frac{t}{2}\Big)  $$
So, by simple computation
\begin{eqnarray*}
\tanh \Big(\frac{t}{2}\Big)&=& \sum_{s\geqslant 0}\frac{(\frac{t}{2})^{2s+1}}{(2s+1)!}.\sum_{m\geqslant 0}(-1)^{m}E_{2m}\frac{(\frac{t}{2})^{2m}}{(2m)!}\\
&=&\sum_{s,m\geqslant 0}\frac{(-1)^{m}}{2^{2m+2s+1}}\frac{E_{2m}t^{2m+2s+1}}{(2m)!(2s+1)!}\\
&=&\sum_{n\geqslant 1}\sum_{m=0}^{n-1}{2n-1 \choose 2m}\frac{(-1)^{m}E_{2m}t^{2n-1}}{2^{2n-1}(2n-1)!},
\end{eqnarray*}
we obtain for $n\geqslant 1$,
$$G_{2n}=\sum_{k=0}^{n-1}(-1)^{n-k-1}(n-k){2n \choose 2k}\frac{E_{2k}}{2^{2n-2}} $$
where $E_{k}$ are Euler numbers. Also the Genocchi numbers $G_{n}$ are defined by the generating function
\begin{equation*}
G(t)=\frac{2t}{e^{t}+1}=\sum_{n=0}^{\infty}G_{n}\frac{t^{n}}{n!}, (|t|< \pi).
\end{equation*}
In general, it satisfies $G_{0}=0,  G_{1}=1,
G_{3}=G_{5}=G_{7}=...G_{2n+1}=0$, and even coefficients are given
$G_{2n}=2(1-2^{2n})B_{2n}=2nE_{2n-1}$, where $B_{n}$ are
Bernoulli numbers and $E_{n}$ are Euler numbers. The first few
Genocchi numbers for even integers are -1,1,-3,17,-155,2073,... .
The first few prime Genocchi numbers are -3 and 17, which occur
at n=6 and 8. There are no others with $n<10^{5}$. For
$x\in\mathbb{R}$, we consider the Genocchi polynomials as follows
\begin{equation*}
G(x,t)=G(t)e^{xt}=\frac{2t}{e^{t}+1}e^{xt}=\sum_{n=0}^{\infty}G_{n}(x)\frac{t^{n}}{n!}.
\end{equation*}
In special case $x=0$, we define $G_{n}(0)=G_{n}$. Because we have
\begin{equation*}
G_{n}(x)=\sum_{k=0}^{n}{n \choose k} G_{k}x^{n-k}.
\end{equation*}
It is easy to deduce that $G_{k}(x)$ are polynomials of degree $k$. Here, we present some of the first Genocchi's polynomials:
$$G_{1}(x)=1,  G_{2}(x)=2x-1, G_{3}(x)=3x^{2}-3x, G_{4}(x)=4x^{3}-6x^{2}+1,   $$

$$G_{5}(x)=5x^{4}-10x^{3}+5x , G_{6}(x)=6x^{5}-15x^{4}+15x^{2}-3, \ ...  $$
The classical Bernoulli polynomials (of higher order)
$B_{n}^{(\alpha)}(x)$ and Euler polynomials (of higher order)
$E_{n}^{(\alpha)}(x), (\alpha\in\mathbb{C})$, are usually defined
by means of the following generating functions
\cite{Ki2,Li2,Lu1,Mc,Ru,Ry1,Sr}.
\begin{equation*}
\Big(\frac{z}{e^{z}-1}\Big)^{\alpha}
e^{xz}=\sum_{n=0}^{\infty}B_{n}^{(\alpha)}(x)\frac{z^{n}}{n!},
(|z| < 2 \pi )
\end{equation*}
and
\begin{equation*}
\Big(\frac{2}{e^{z}+1}\Big)^{\alpha}e^{xz}=\sum_{n=0}^{\infty}E_{n}^{(\alpha)}(x)\frac{z^{n}}{n!},
(|z| < \pi).
\end{equation*}
So that, obviously,
\begin{equation*}
B_{n}(x):=B_{n}^{1}(x) \ \ \ \text{and} \ \ \ E_{n}(x):= E_{n}^{(1)}(x).
\end{equation*}
In 2002, Q. M. Luo and et al. (see \cite{Gu,Lu3,Lu4}) defined the
generalization of Bernoulli polynomials and Euler numbers, as
follows
\begin{equation*}
\frac{tc^{xt}}{b^{t}-a^{t}}=\sum_{n=0}^{\infty}\frac{B_{n}(x;a,b,c)}{n!}t^{n},
        (| t\ln \frac{b}{a}|<2 \pi)
\end{equation*}
\begin{equation*}
\frac{2}{b^{t}+a^{t}}=\sum_{n=0}^{\infty}E_{n}(a,b)\frac{t^{n}}{n!}
, (| t\ln \frac{b}{a}|< \pi).
\end{equation*}
Here, we give an analogous definition for generalized Apostol-Genocchi polynomials.

Let $a, b> 0$, The Generalized Apostol-Genocchi Numbers and
Apostol-Genocchi polynomials with $a, b, c$ parameters are defined
by
\begin{equation*}
\frac{2t}{\lambda b^{t}+a^{t}}= \sum_{n=0}^{\infty}G_{n}(a, b;\lambda)\frac{t^{n}}{n!}
\end{equation*}
\begin{equation*}
\frac{2t}{\lambda b^{t}+a^{t}}e^{xt}= \sum_{n=0}^{\infty}G_{n}(x, a, b;\lambda)\frac{t^{n}}{n!}
\end{equation*}
\begin{equation*}
\frac{2t}{\lambda b^{t}+a^{t}}c^{xt}= \sum_{n=0}^{\infty}G_{n}(x, a, b,c;\lambda)\frac{t^{n}}{n!}
\end{equation*}
respectively.

For a real or complex parameter $\alpha$, The Apostol-Genocchi
polynomials with $a, b, c$ parameters of order $\alpha$,
$G_{n}^{(\alpha)}(x; a, b; \lambda)$, each of degree $n$ is $x$
as well as in $\alpha$, are defined by the following generating
functions
\begin{equation*}
\Big(\frac{2t}{\lambda
b^{t}+a^{t}}\Big)^{\alpha}e^{xz}=\sum_{n=0}^{\infty}G_{n}^{(\alpha)}(x,
a, b;\lambda)\frac{t^{n}}{n!}.
\end{equation*}
Clearly, we have $G_{n}^{(1)}(x, a, b; \lambda)=G_{n}(x; a, b; \lambda)$.

Now, we introduce the 2-variable Apostol-Genocchi Polynomials and
then we consider the multiplication theorem for 2-variable
Apostol-Genocchi Polynomials.

Now, we start with the definition of Apostol-Genocchi Polynomials $G_{n}(x ; \lambda)$. The Apostol-Genocchi Polynomials $G_{n}(x; \lambda)$ in variable $x$ are defined by means of the generating function
\begin{equation*}
\frac{2ze^{xz}}{\lambda  e^{z}+1}=\sum_{n=0}^{\infty}G_{n}(x; \lambda)\frac{z^{n}}{n!}\ \ \ (|z|< 2\pi \ when \ \lambda =1, |z|< |\log\lambda| \ when \ \lambda \neq 1).
\end{equation*}
with, of course
\begin{equation*}
G_{n}(\lambda):= G_{n}(0; \lambda)
\end{equation*}
Where $G_{n}(\lambda)$ denotes the so-called Apostol-Genocchi numbers.

Also (see \cite{Ap,Li3,Lu2,Lu5,Lu6,Ru,Sr}) Apostol-Genocchi Polynomials $G_{n}^{(\alpha)}(x; \lambda)$ of order $\alpha$ in variable $x$ are defined by means of the generating function:
\begin{equation*}
\Big(\frac{2z}{\lambda e^{z}+1}\Big)^{\alpha}
e^{xz}=\sum_{n=0}^{\infty}G_{n}^{(\alpha )}(x;
\lambda)\frac{z^{n}}{n!}
\end{equation*}
 with, of course $G_{n}^{(\alpha)}(\lambda):= G_{n}^{\alpha}(0; \lambda)$.

 Where $G_{n}^{\alpha }(\lambda)$ denotes the so-called Apostol-Genocchi numbers of higher order.
  If we set
 $$\phi(x,t;\alpha)=\Big(\frac{2t}{e^{t}+1}\Big)^{\alpha}e^{xt} $$
 Then
 $$\frac{\partial\phi}{\partial x}=t\phi $$
 and
 $$t\frac{\partial\phi}{\partial t}-\Big\{\frac{\alpha+tx}{t}-\frac{\alpha
e^{t}}{e^{t}+1}\Big\}\frac{\partial\phi}{\partial x}=0. $$

Next, we introduce the class of Apostol-Genocchi numbers as
follows. (for more information see [38])
\begin{equation*}
_{H}G_{n}(\lambda)=\sum_{s=0}^{[\frac{n}{2}]}\frac{n!G_{n-2s}(\lambda)G_{s}(\lambda)}{s!(n-2s)!}
\end{equation*}
The generating function of $_{H}G_{n}(\lambda)$ is provided by
\begin{equation*}
\frac{4t^{3}}{(\lambda e^{t}+1) (\lambda e^{t^{2}}+1)}=\sum_{n=0}^{\infty}\ _{H}G_{n}(\lambda )\frac{t^{n}}{n!}
\end{equation*}
and the generalization of $_{H}G_{n}(\lambda)$ for $(a,b)\neq 0$, is
\begin{equation*}
\frac{4t^{3}}{(\lambda e^{at}+1) (\lambda e^{bt^{2}}+1)}=\sum_{n=0}^{\infty}\ _{H}G_{n}(a, b;\lambda )\frac{t^{n}}{n!}
\end{equation*}
where
\begin{equation*}
_{H}G_{n}(a, b;\lambda)=\frac{1}{ab}\sum_{n=0}^{[\frac{n}{2}]}\frac{n!a^{n-2s}b^{s}G_{n-2s}(\lambda)G_{s}(\lambda)}{s!(n-2s)!}
\end{equation*}
The main object of the present paper is to investigate the multiplication formulas for the Apostol-type polynomials.

Luo in \cite{Lu2} defined the multiple alternating sums as
\begin{equation*}
Z_{k}^{(l)}(m;\lambda)=(-1)^{l}\sum_{^{0\leq v_{1},v_{2},...,v_{m}\leq l}_{v_{1}+v_{2}+...+v_{m}=\ell}}{l \choose v_{1},v_{2},...,v_{m}}(-\lambda)^{v_{1}+2v_{2}+...+mv_{m}}
\end{equation*}
\begin{equation*}
Z_{k}(m;\lambda)=\sum_{j=1}^{m}(-1)^{j+1}\lambda^{j}j^{k}=\lambda -\lambda^{2}2^{k}+...+(-1)^{m+1}\lambda^{m}m^{k}
\end{equation*}
\begin{equation*}
Z_{k}(m)=\sum_{j=1}^{m}(-1)^{j+1}j^{k}=1-2^{k}+...+(-1)^{m+1}m^{k}, \ (m,k,l\in\mathbb{N}_{0}; \lambda\in\mathbb{C})
\end{equation*}
where $\mathbb{N}_{0}:=\mathbb{N}\cup\{0\} \  , \ (\mathbb{N}:=\{1,2,3,... \}). $

\section{the multiplication formulas for the apostol-genocchi polynomials of higher order}

In this Section, we obtain some interesting new relations and
properties associated with Apostol-Genocchi polynomials of higher
order and then derive several elementary properties including
recurrence relations for Genocchi numbers. First of all we prove
the multiplication theorem of these polynomials

\begin{theorem}\label{t1}
For $m\in\mathbb{N}$, $n\in\mathbb{N}_{0}$, $\alpha, \lambda\in\mathbb{C} $, the following multiplication formula of the Apostol-Genocchi polynomials of higher order holds true:
\begin{equation}\label{8}
G_{n}^{(\alpha)}(mx;\lambda)=m^{n-\alpha}\sum_{v_{1},v_{2},...,v_{m-1}\geq
0}{\alpha \choose
v_{1},v_{2},...,v_{m-1}}(-\lambda)^{r}G_{n}^{(\alpha)}\Big(x+\frac{r}{m};\lambda^{m}\Big)
\end{equation}
where $r=v_{1}+2v_{2}+...+(m-1)v_{m-1}$, ($m$ is odd )
\end{theorem}
\begin{proof}
It is easy to observe that
\begin{equation*}
\frac{1}{\lambda e^{t}+1}=- \frac{1-\lambda
e^{t}+\lambda^{2}e^{2t}+...+(-\lambda)^{m-1}e^{(m-1)t}}{(-\lambda)^{m}
e^{mt}-1}.   \qquad (\ast)
\end{equation*}
But we have, if $x_{i}\in\mathbb{C}$
\begin{equation*}
(x_{1}+x_{2}+...+x_{m})^{n}=\sum_{^{a_{1},a_{2},...,a_{m}\geqslant
0}_{a_{1}+a_{2}+...a_{m}=n}}{n \choose a_{1},a_{2},...,a_{m}
}x_{1}^{a_{1}}x_{2}^{a_{2}}...x_{m}^{a_{m}}.\qquad (\ast \ast)
\end{equation*}
The last summation takes place over all positive or zero integers $a_{i}\geqslant 0$ such that $a_{1}+a_{2}+...+a_{m}=n$, where
\begin{equation*}
{n \choose a_{1},a_{2},...,a_{m}}:=\frac{n!}{a_{1!}a_{2}!...a_{m}!}.
\end{equation*}
So by applying $(\ast,\ast\ast)$, we get
\begin{eqnarray*}
\sum_{n=0}^{\infty}G_{n}^{(\alpha)}(mx;\lambda)\frac{t^{n}}{n!}&=&\Big(\frac{2t}{\lambda e^{t}+1}\Big)^{\alpha}e^{mxt} \\
   &=&\Big(\frac{2t}{\lambda^{m} e^{mt}+1}\Big)^{\alpha}\Big(\sum_{k=0}^{m-1}(-\lambda)^{k}e^{kt}\Big)e^{mxt} \\
  &=&\sum_{v_{1},v_{2},...,v_{m-1}\geqslant0}{\alpha \choose v_{1},v_{2},...,v_{m-1}}(-\lambda)^{r}
  \Big(\frac{2t}{\lambda^{m}e^{mt}+1}\Big)^{\alpha}e^{(x+\frac{r}{m})mt}  \\
&=&\sum_{n=0}^{\infty}\bigg(m^{n-\alpha}\sum_{v_{1},v_{2},...,v_{m}\geqslant
0}{^{\alpha} \choose
v_{1},v_{2},...,v_{m}}(-\lambda)^{r}G_{n}^{(\alpha)}\Big(x+\frac{r}{m};\lambda^{m}\Big)\bigg)\frac{t^{n}}{n!}.
\end{eqnarray*}
By comparing the coefficient of $\frac{t^{n}}{n!}$ on both sides of last equation, proof is complete.
\end{proof}
In terms of the generalized Apostol-Genocchi polynomials, by
setting $\lambda= 1$ in Theorem 2.1, we obtain the following
explicit formula that is called multiplication Theorem for
Genocchi polynomials of higher order.

\begin{corollary}\label{co1}
For $m\in\mathbb{N}$, $n\in\mathbb{N}_{0}$, $\alpha, \in\mathbb{C}
$, we have
$$G_{n}^{(\alpha)}(mx)=m^{n-\alpha}\sum_{v_{1},v_{2},...,v_{m-1}\geqslant 0}{\alpha \choose v_{1},v_{2},...,v_{m-1}}(-1)^{r}G_{n}^{(\alpha)}\Big(x+\frac{r}{m}\Big) \ \ \text{(m is odd)}. $$
\end{corollary}
And using corollary 2.2, (by setting $\alpha = 1$), we get
Corollary 2.3 that is the main result of [36] and is called
multiplication Theorem for Genocchi polynomials.
\begin{corollary}\label{co2}
For $m\in\mathbb{N}$, $n\in\mathbb{N}_{0}$, we have
$$G_{n}(mx)=m^{n-1}\sum_{k=0}^{m-1}(-1)^{k}G_{n}\Big(x+\frac{k}{m}\Big) \ \ \text{(m is odd)}.  $$
\end{corollary}
Now, we consider the multiplication formula for the Apostol-Genocchi numbers when, $m$ is even.
\begin{theorem}\label{t2}
For $m\in\mathbb{N}$ (m even), $n\in\mathbb{N}$, $\alpha, \lambda\in\mathbb{C}$, the following multiplication formula of the Apostol-Genocchi polynomials of higher order holds true:
$$ G_{n}^{(\alpha)}(mx;\lambda)=(-2)^{\alpha}m^{n-\alpha}\sum_{v_{1},v_{2},...,v_{m-1}\geqslant 0}{\alpha \choose v_{1},v_{2},...,v_{m-1}}(-\lambda)^{r}B_{n}^{(\alpha)}\Big(x+\frac{r}{m},\lambda^{m}\Big),$$
where $r=v_{1}+2v_{2}+...+(m-1)v_{m-1}.$
\end{theorem}
\begin{proof}
It is easy to observe that
$$\frac{1}{\lambda e^{t}+1}=-\frac{1-\lambda e^{t}+\lambda^{2}e^{2t}+...+(-\lambda)^{m-1}e^{(m-1)t}}{(-\lambda)^{m}e^{mt}-1}. $$
So, we obtain
\begin{eqnarray*}
\sum_{n=0}^{\infty}G_{n}^{(\alpha)}(mx;\lambda)\frac{t^{n}}{n!} &=& \Big(\frac{2t}{\lambda e^{t}+1}\Big)^{\alpha}e^{mxt} \\
&=&2^{\alpha}\Big(\frac{t}{\lambda e^{t}+1}\Big)^{\alpha}e^{mxt} \\
&=&(-2)^{\alpha}\Big(\frac{t}{\lambda^{m}e^{mt}-1}\Big)^{\alpha} \Big(\sum_{k=0}^{m-1}(-\lambda e^{t})^{k}\Big)^{\alpha}e^{mxt} \\
& = &(-2)^{\alpha}\sum_{v_{1},v_{2},...,v_{m-1}\geqslant 0}{\alpha \choose v_{1},v_{2},...,v_{m-1}}(-\lambda)^{r}\Big(\frac{t}{\lambda^{m} e^{m} -1}\Big)^{\alpha}e^{(x+\frac{r}{m})mt} \\
&=&\sum_{n=0}^{\infty}\Big((-2)^{\alpha}m^{n-\alpha}\sum_{v_{1},v_{2},...,v_{m-1}\geqslant 0}{\alpha \choose v_{1},v_{2},...,v_{m-1}}(-\lambda)^{r} \\
& \times &
B_{n}^{(\alpha)}(x+\frac{r}{m};\lambda^{m})\Big)\frac{t^{n}}{n!}
\end{eqnarray*}
By comparing the coefficients of $\frac{t^{n}}{n!}$ on both sides proof will be complete.
\end{proof}
Next, using Theorem 2.4, (with $ \lambda= 1$), we obtain the
Genocchi polynomials of higher order can be expressed by the
Bernoulli polynomials of higher order when, $m$ is even
\begin{corollary}\label{co3}
For $m\in\mathbb{N}$ (m even), $n\in\mathbb{N}_{0}$, $\alpha
\in\mathbb{C}$, we get
$$G_{n}^{(\alpha)}(mx)=(-2)^{\alpha}m^{n-\alpha}\sum_{v_{1},v_{2},...,v_{m-1}\geqslant 0}{\alpha \choose v_{1},v_{2},...,v_{m-1}}(-1)^{r}B_{n}^{\alpha}\Big(x+\frac{r}{m}\Big). $$

\end{corollary}
Also by applying $\alpha =1$, in corollary 2.5 we obtain the
following assertion that is one of the most remarkable identities
in area of Genocchi polynomials.
\begin{corollary}\label{co4}
For $m\in\mathbb{N}$, $n\in\mathbb{N}_{0}$, we obtain
$$G_{n}(mx)=-2m^{n-1}\sum_{k=0}^{m-1}(-1)^{k}B_{n}\Big(x+\frac{k}{m}\Big) \ \ \text{ $m$ is even}. $$
\end{corollary}
Obviously, the result of Corollary \ref{co4} is analogous with the
well-known Raabe's multiplication formula. Now, we present
explicit evaluations of $Z_{n}^{(l)}(m;\lambda)$,
$Z_{n}^{(l)}(\lambda)$, $Z_{n}(m)$ by Apostol-Genocchi
polynomials.
\begin{theorem}\label{t3}
For $m,n,l\in\mathbb{N}_{0},\lambda\in\mathbb{C}$, we have
$$Z_{n}^{(l)}(m;\lambda)=2^{-l}\sum_{j=0}^{l}{l \choose j}\frac{(-1)^{j(m+1)}\lambda^{mj+l} }{(n+1)_{l}}\sum_{k=0}^{n+l}{n+l \choose k}G_{k}^{(j)}(mj+l;\lambda)G_{n+l-k}^{(l-j)}(\lambda) $$
where $(n)_{0}=1, (n)_{k}=n(n+1)...(n+k-1). $
\end{theorem}
\begin{proof}
By definition of $Z_{n}^{(l)}(m;\lambda)$, we calculate the following sum
\\
\begin{flushleft}
    $\sum_{n=0}^{\infty}Z_{n}^{(l)}(m;\lambda)\frac{t^{n}}{n!}=$
\end{flushleft}
$$\sum_{n=0}^{\infty}\Big[(-1)^{l}\sum_{^{0\leqslant v_{1},v_{2},...,v_{m}\leqslant l}_{v_{1}+v_{2}+...+v_{m}=l}}{l \choose v_{1},v_{2},...,v_{m}}(-\lambda)^{\lambda_{1}+2\lambda_{2}+...+m\lambda_{m}}(v_{1}+2v_{2}+...+mv_{m})^{n}\Big]\frac{t^{n}}{n!} $$
\begin{eqnarray*}
&=&(-1)^{l}\sum_{^{0\leqslant v_{1},v_{2},...,v_{m}\leqslant l}_{v_{1}+v_{2}+...+v_{m}=l}}{l \choose v_{1},v_{2},...,v_{m}}(-\lambda e^{t})^{\lambda_{1}+2\lambda_{2}+...+m\lambda_{m}} \\
&=&(\lambda e^{t}-\lambda^{2}e^{2t}+...+(-1)^{m+1}\lambda^{m}e^{mt})^{l} \\
&=&\bigg(\frac{(-1)^{m+1}\lambda^{m+1}e^{(m+1)t}}{\lambda e^{t}+1}+\frac{\lambda e^{t}}{\lambda e^{t}+1}\bigg)^{l}\\
&=&(2t)^{-l}\sum_{j=0}^{l}{l \choose j}\Big[\frac{2t(-1)^{m+1}\lambda^{m+1}e^{(m+1)t}}{\lambda e^{t} +1}\Big]^{j}\Big[\frac{2t\lambda e^{t}}{\lambda e^{t}+1}\Big]^{l-j}\\
&=&(2t)^{-l}\sum_{j=0}^{l}{l \choose j}(-1)^{j(m+1)}\lambda^{mj+l}\sum_{n=0}^{\infty}G_{n}^{(j)}(mj+l;\lambda)\frac{t^{n}}{n!}
\sum_{n=0}^{\infty}G_{n}^{(l-j)}(\lambda)\frac{t^{n}}{n!}\\
&=&2^{-l}\sum_{n=0}^{\infty}\bigg[\sum_{j=0}^{l}{j \choose
l})\frac{(-1)^{j(m+1)}\lambda^{mj+l}}{(n+1)_{l}}
\sum_{k=0}^{n+l}{n+l \choose
k}G_{k}^{(j)}(mj+l;\lambda)G_{n+l-k}^{(l-j)}(\lambda)\bigg]\frac{t^{n}}{n!}
\end{eqnarray*}
by comparing the coefficients of $\frac{t^{n}}{n!}$ on both sides, proof will be complete.
\end{proof}

As a direct result, using $\lambda=1$ in Theorem 2.7, we derive an
explicit representation of multiple alternating sums $Z^{(l)}_n
(m)$, in terms of the Genocchi polynomials of higher order. We
also deduce their special cases and applications which lead to the
corresponding results for the Genocchi polynomials.

\begin{corollary}
Form $m,n,l\in\mathbb{N}_{0}$, the following formula holds true in terms of the Genocchi polynimials
\begin{equation*}
Z_{n}^{(l)}(m)=2^{-l}\sum_{j=0}^{l}{l \choose j}\frac{(-1)^{j(m+1)}}{(n+1)_{l}}\sum_{k=0}^{n+l}{n+l \choose k}G_{k}^{(j)}(mj+l)G_{n+l-k}^{l-j}
\end{equation*}
where $(n)_{0}=1, (n)_{k}=n(n+1)...(n+k-1)$.
\end{corollary}

Next we investigate some of the recursive formulas for the
Apostol-Genocchi numbers of higher order that are analogous to the
results of Howard \cite{Ho1,Ho2,Ho3} and we deduce that they
constitute a useful special case.
\begin{theorem}\label{t4}
For $m$ be odd, $n,l\in\mathbb{N}_{0} \ ,\lambda\in\mathbb{C}$, we have
\end{theorem}
\begin{equation*}
m^{n}G_{n}^{(l)}(\lambda^{m})-m^{l}G_{n}^{(l)}(\lambda)=(-1)^{l-1}\sum_{k=0}^{n}{n \choose k}m^{k}G_{k}^{(l)}(\lambda^{m})Z_{n-k}^{(l)}(m-1;\lambda).
\end{equation*}
\begin{proof}
By taking $x=0, \alpha=l$ in (\ref{8}), where $r=v_{1}+2v_{2}+...+(m-1)v_{m-1}$ we obtain
\begin{equation*}
m^{l}G_{n}^{(l)}(\lambda)=m^{n}\sum_{v_{1},v_{2},...,v_{m-1}\geqslant 0}{l \choose v_{1},v_{2},...,v_{m-1}}(-\lambda)^{r}G_{n}^{(l)}(\frac{r}{m},\lambda^{m}).
\end{equation*}
But we know
\begin{equation*}
G_{n}^{(l)}(x;\lambda)=\sum_{k=0}^{n}{n \choose k}G_{k}^{(l)}(\lambda)x^{n-k}.
\end{equation*}
So, we obtain
\begin{eqnarray*}
m^{l}G_{n}^{(l)}(\lambda)&=&m^{n}\sum_{v_{1},v_{2},...,v_{m-1}\geqslant  0}{l \choose v_{1},v_{2},...,v_{m-1}}(-\lambda)^{r}\sum_{k=0}^{n}{n \choose k}G_{k}^{(l)}(\lambda^{m})\Big(\frac{r}{m}\Big)^{n-k} \\
&=&\sum_{k=0}^{n}{n \choose k}m^{k}G_{k}^{(l)}(\lambda^{m})\sum_{0 \leqslant v_{1},v_{2},...,v_{m-1}\leqslant l}{l \choose v_{1},v_{2},...,v_{m-1}}(-\lambda)^{r}r^{n-k}\\
&=&\sum_{k=0}^{n}{n \choose k}m^{k}G_{k}^{(l)}(\lambda^{m})\sum_{^{0 \leqslant v_{1},v_{2},...,v_{m-1}\leqslant l}_{v_{1}+v_{2}+...v_{m-1}=l}}{l \choose v_{1},v_{2},...,v_{m-1}}(-\lambda)^{r}r^{n-k}+m^{n}G_{n}^{(l)}(\lambda^{m}) \\
&=&(-1)^{l}\sum_{k=0}^{n}{n \choose k}m^{k}G_{k}^{(l)}(\lambda^{m})Z_{n-k}^{(l)}(m-1;\lambda)+m^{n}G_{n}^{(l)}(\lambda^{m})
\end{eqnarray*}
So proof is complete.
\end{proof}
Furthermore, we derive some well-known results (see [10])
involving Genocchi polynomials of higher order and Genocchi
polynomials which we state here. By setting $\lambda=1$, $l=1$ in
Theorem 2.9, we get Corollaries 2.10, 2.11, respectively
\begin{corollary}\label{co5}
For $m$ be odd, $n,l\in\mathbb{N}_{0} \ $, we have
\begin{equation*}
(m^{n}-m^{l})G_{n}^{(l)}=(-1)^{l-1}\sum_{k=0}^{n}{n \choose k}G_{k}^{(l)}Z_{n-k}^{(l)}(m-1).
\end{equation*}
\end{corollary}
\begin{corollary}
For $m$ be odd, $n\in\mathbb{N}_{0} \ ,\lambda\in\mathbb{C}$, we
have
\begin{equation*}
m^{n}G_{n}(\lambda^{m})-mG_{n}(\lambda)=\sum_{k=0}^{n}{n \choose k}m^{k}G_{k}(\lambda^{m})Z_{n-k}(m-1;\lambda).
\end{equation*}
\end{corollary}
Also by setting $\lambda=1$ in Corollary 2.10, we get following
assertion that is analogous to the formula of Howard in terms of
Genocchi numbers.
\begin{corollary}
For $m$ be odd, $n,l\in\mathbb{N}_{0} \ ,\lambda\in\mathbb{C}$,
we obtain
$$(m^{n}-m)G_{n}=\sum_{k=0}^{n}{n \choose k}m^{k}G_{k}Z_{n-k}(m-1) $$.
\end{corollary}
Next, we investigate the generalization of Howard's formula in terms of Apostol-Genocchi numbers, when $m$ is even.
\begin{theorem}\label{t5}
For $m$ be even, $n,l\in\mathbb{N}_{0}, \ \lambda\in\mathbb{C}$, the following formula
$$m^{l}G_{n}^{(l)}(\lambda)-(-2)^{l}m^{n}B_{n}^{(l)}(\lambda^{m})=2^{l}\sum_{k=0}^{n}{n \choose k}m^{k}B_{k}^{(l)}(\lambda^{m})Z_{n-k}^{(l)}(m-1;\lambda) $$
holds true, where $r=v_{1}+2v_{2}+...+(m-1)v_{m-1}$.
\end{theorem}
\begin{proof}
We have
$$G_{n}^{(l)}(\lambda)=(-2)^{l}m^{n-l}\sum_{v_{1},v_{2},...,v_{m-1}\geqslant 0}{l \choose v_{1},v_{2},...,v_{m-1}}(-\lambda)^{r}B_{n}^{(l)}(\frac{r}{m},\lambda^{m}) $$
But we know
$$B_{n}^{(l)}(x;\lambda)=\sum_{k=0}^{n}{n \choose k}B_{k}^{(l)}(\lambda)x^{n-k}. $$
So we get

\begin{eqnarray}
  m^{l}G_{n}^{(l)}(\lambda) &=& (-2)^{l}m^{n}\sum_{v_{1},v_{2},...,v_{m-1}\geqslant 0}{l \choose v_{1},v_{2},...,v_{m-1}}(-\lambda)^{r}\sum_{k=0}^{n}{n \choose k}B_{k}^{(l)}(\lambda^{m})\Big(\frac{r}{m}\Big)^{n-k} \nonumber \\
   &=& (-2)^{l}\sum_{k=0}^{n}{n \choose k}m^{k}B_{k}^{(l)}(\lambda^{m})\sum_{v_{1},v_{2},...,v_{m-1}\geqslant 0}{l \choose v_{1},v_{2},...,v_{m-1}}(-\lambda)^{r}r^{n-k} \nonumber \\
   &=& 2^{l}\sum_{k=0}^{n}{n \choose k}m^{k}B_{k}^{(l)}(\lambda^{m})Z_{n-k}^{(l)}(m-1;\lambda) \nonumber+(-2)^{l}m^{n}B_{n}^{(l)}(\lambda^{m}) \nonumber
\end{eqnarray}
So we obtain
$$m^{l}G_{n}^{(l)}(\lambda)-(-2)^{l}m^{n}B_{n}^{(l)}(\lambda^{m})
 =2^{l}\sum_{k=0}^{n}{n \choose k}m^{k}B_{k}^{(l)}(\lambda^{m})Z_{n-k}^{(l)}(m-1;\lambda) $$
 So proof is complete.
 \end{proof}
 Also by letting $\lambda=1$ in Theorem \ref{t5}, we obtain
 following assertion
\begin{corollary}
For $m$ be even, $n,l\in\mathbb{N}_{0}$, we get
$$m^{l}G_{n}^{(l)}-(-2)^{l}m^{n}B_{n}^{(l)}=2^{l}\sum_{k=0}^{n}{n \choose k}m^{k}B_{n}^{(l)}Z_{n-k}^{(l)}(m-1) $$
\end{corollary}
Here we present the lowering orders for Apostol-Genocchi numbers of higher order.
\begin{theorem}
(Lowering orders) For $n,k\geqslant 1$,
$$G_{k}^{(n+1)}(\lambda)=2kG_{k-1}^{(n)}(\lambda)-\Big(2-\frac{2k}{n}\Big)G_{k}^{(n)}(\lambda) $$
\end{theorem}
\begin{proof}
Let us put $G_{n}(t;\lambda)=\Big(\frac{2t}{\lambda
e^{t}+1}\Big)^{n}$. Then $G_{n}(t;\lambda)$ is the generating
function of higher order Apostol-Genocchi numbers. The derivative
$G^{'}(t;\lambda)=\frac{d}{dt}G_{n}(t;\lambda)$ is equal to
\begin{equation*}
n\Big(\frac{1}{t}-\frac{\lambda e^{t}}{\lambda
e^{t}+1}\Big)G_{n}(t;\lambda)
=\frac{n}{t}G_{n}(t;\lambda)-nG_{n}(t;\lambda)+\frac{n}{\lambda
e^{t}+1}G_{n}(t;\lambda)
\end{equation*}
and
\begin{equation*}
tG_{n}^{'}(t;\lambda)=nG_{n}(t;\lambda)-ntG_{n}(t;\lambda)+\frac{n}{2}G_{n+1}(t)
\end{equation*}
so we obtain
\begin{equation*}
\frac{G_{k}^{(n)}(\lambda)}{(k-1)!}=n\frac{G_{k}^{(n)}(\lambda)}{k!} -n\frac{G_{k-1}^{(n)}(\lambda)}{(k-1)!}+\frac{n}{2}\frac{G_{k}^{(n+1)}(\lambda)}{k!}
\end{equation*}

for $k\geqslant 1$. This formula can written as
$$G_{k}^{(n+1)}(\lambda)=2kG_{k-1}^{(n)}(\lambda)-\Big(2-\frac{2k}{n}\Big)G_{k}^{(n)}(\lambda) $$
so proof is complete.
\end{proof}

\section{generalized apostol genocchi polynomials with $a, b,c$ parameters}
In this section we investigate some recurrence formulas for
generalized Apostol-Genocchi polynomials with $a, b,c$ parameters
. In 2003, Cheon [3] rederived several known properties and
relations involving the classical Bernoulli polynomials $B_n (x)$
and the classical Euler polynomials $E_n (x)$ by making use of
some standard techniques based upon series rearrangement as well
as matrix representation. Srivastava and Pinter [36] followed
Cheon's work [3] and established two relations involving the
generalized Bernoulli polynomials $B^{(\alpha)}_n (x)$ and the
generalized Euler polynomials $E^{(\alpha)}_n (x)$. So, we will
study further the relations between generalized Bernoulli
polynomials with $a, b$ parameters and Genocchi polynomials with
the methods of generating function and series rearrangement.

\begin{theorem}\label{t6}
Let $x\in\mathbb{R}$ and $n\geqslant 0$. For every positive real number $a,b$ and $c$ such that $a\neq b$ and $b> 0$, we have
$$G_{n}^{(\alpha)}(a,b;\lambda)=G_{n}^{(\alpha)}\Big(\frac{\alpha\ln a}{\ln a-\ln b}; \lambda\Big)(\ln b-\ln a )^{n-\alpha} $$
\end{theorem}
\begin{proof}
We know
\begin{eqnarray*}
  \Big(\frac{2t}{\lambda b^{t}+a^{t}}\Big)^{\alpha} &=& \sum_{n=0}^{\infty}G_{n}^{(\alpha)}(a,b;\lambda)\frac{t^{n}}{n!}\\
   &=& \frac{1}{a^{\alpha t}}\Big(\frac{2t}{\lambda e^{t(\ln b-\ln a)}+1}\Big)^{\alpha} \\
   &=& e^{-t\alpha\ln a}\Big(\frac{2t(\ln b-\ln a)}{\lambda e^{t(\ln b-\ln a)}+1}\Big)^{\alpha}\times\frac{1}{(\ln b-\ln a)^{\alpha}} \\
   &=& \frac{1}{(\ln b-\ln a)^{\alpha}}\sum_{n=0}^{\infty}G_{n}^{(\alpha)}\Big(\frac{\alpha \ln a}{\ln a-\ln b};\lambda\Big)(\ln b-\ln a)^{n}\frac{t^{n}}{n!}\\
\end{eqnarray*}
So by comparing the coefficients of $\frac{t^{n}}{n!}$ on both sides, we get
$$G_{n}^{(\alpha)}(a,b;\lambda)=G_{n}^{(\alpha)}\Big(\frac{\alpha\ln a}{\ln a-\ln b}; \lambda\Big)(\ln b-\ln a)^{n-\alpha}. $$
\end{proof}
\begin{theorem}
Suppose that the conditions of Theorem \ref{t6} holds true, we get
$$G_{n}^{(\alpha)}(x;a,b,c;\lambda)=G_{n}^{(\alpha)}\Big(\frac{-\alpha\ln a+x\ln c}{\ln b-\ln a},\lambda\Big)(\ln b-\ln a )^{n-\alpha} $$
\end{theorem}
\begin{proof}
We have
\begin{eqnarray*}
  \sum_{n=0}^{\infty}G_{n}^{(\alpha )}(x;a,b,c;\lambda)&=& \Big(\frac{2t}{\lambda b^{t}+a^{t}}\Big)^{\alpha}c^{xt} \\
&=&\frac{1}{\alpha^{at}}\Big(\frac{2t}{\lambda e^{t(\ln b-\ln a)}+1}\Big)^{\alpha}c^{xt} \\
&=& e^{t(-\alpha\ln a + x\ln c)}\Big(\frac{2t}{\lambda e^{t(\ln b-\ln a)}+1}\Big)^{\alpha} \\
&=&\frac{1}{(\ln b-\ln a
)^{\alpha}}\sum_{n=0}^{\infty}G_{n}^{(\alpha)}\Big(\frac{-\alpha\ln
a+x\ln c}{\ln b-\ln a},\lambda\Big)(\ln b-\ln
a)^{n}\frac{t^{n}}{n!}.
\end{eqnarray*}

So by comparing the coefficient of $\frac{t^{n}}{n!}$ on both sides, we get

$$G_{n}^{(\alpha)}(x;a,b,c;\lambda)=G_{n}^{(\alpha)}\Big(\frac{-\alpha\ln a+x\ln c}{\ln b-\ln a},\lambda\Big)(\ln b-\ln a)^{n-\alpha}. $$
Therefore  proof is complete.
\end{proof}
The generalized Apostal-Genocchi polynomials of higher order $G_{n}^{(\alpha)}(x;a,b,c;\lambda)$ prossess a number of interesting properties which we state here.
\begin{theorem}\label{t7}
Let $a,b,c\in\mathbb{R}^{+} \  (a\neq b)$ and $x\in\mathbb{R}$, then
\begin{equation}\label{1}
   G_{n}^{(\alpha)}(x+1;a,b,c;\lambda) = \sum_{k=0}^{n}{n \choose k}(\ln c)^{n-k}G_{k}^{(\alpha)}(x;a,b,c;\lambda)
\end{equation}
\begin{equation}\label{2}
  G_{n}^{(\alpha)}(x+\alpha;a,b,c;\lambda) = G_{n}^{(\alpha)}\Big(x;\frac{a}{c},\frac{b}{c},c;\lambda\Big)
\end{equation}
\begin{equation}\label{3}
  G_{n}^{(\alpha)}(\alpha-x;a,b,c;\lambda)=G_{n}^{(\alpha)}\Big(-x;\frac{a}{c},\frac{b}{c},c;\lambda\Big)
\end{equation}
\begin{equation}\label{4}
  G_{n}^{(\alpha +\beta)}(x+y;a,b,c;\lambda)=\sum_{r=0}^{k}{k \choose r}G_{k-r}^{(\alpha)}(x;a,b,c;\lambda)G_{r}^{(\beta)}(y;a,b,c;\lambda)
\end{equation}
\begin{equation}\label{5}
   \frac{\partial^{l}}{\partial x^{l}}\{ G_{n}^{(\alpha)}(x;a,b,c;\lambda) \}= \frac{n!}{(n-\ell)!}(\ln c)^{\ell}G_{n-\ell}^{(\alpha)}(x;a,b,c;\lambda)
\end{equation}
\begin{equation}\label{6}
   \int^{t}_{s}G_{n}^{(\alpha)}(x;a,b,c;\lambda)dx= \frac{1}{(n+1)\ln c}\Big[G_{n+1}^{(\alpha)}(t;a,b,c;\lambda)-G_{n+1}^{(\alpha)}(s;a,b,c;\lambda)\Big]
\end{equation}
\end{theorem}
\begin{proof}
We know
\begin{eqnarray*}
  \sum_{n=0}^{\infty}G_{n}^{(\alpha)}(x+1;a,b,c;\lambda)\frac{t^{n}}{n!} &=& \Big(\frac{t}{\lambda b^{t}+ a^{t}}\Big)^{\alpha}.c^{(x+1)t} \\
  &=& \sum_{n=0}^{\infty}\sum_{k=0}^{\infty}G_{k}^{(\alpha)}(x;a,b,c;\lambda)(\ln c)^{n}\frac{t^{n+k}}{n!k!}\\
  &=&  \sum_{n=0}^{\infty}\sum_{k=0}^{\infty}G_{k}^{(\alpha)}(x;a,b,c;\lambda)(\ln c)^{n-k}\frac{t^{n+k}}{(n-k)!k!}
\end{eqnarray*}
So comparing the coefficients of $t^{n}$ on both sides, we arrive at the result (\ref{1}) asserted by Theorem \ref{t7}.
Similary, by simple manipulations, leads us to the result (\ref{2}), (\ref{3}) and (\ref{4}) of Theorem \ref{t7} and by successive differentiation with respect to $x$ and then using the principle of mathematical induction on $\ell\in\mathbb{N}_{0}$, we obtain the formula (\ref{5}). Also, by taking $\ell=1$ in (\ref{5}) and integrating both sides with respect to $x$, we get the formula (\ref{6}).
\end{proof}
\begin{remark}
Let $a,b,c\in\mathbb{R}^{+} \ (a\neq -b)$ and $x\in\mathbb{R}$, by differentiating both sides of the following generating function
$$\sum_{n=0}^{\infty}G_{n}^{\alpha}(x;a,b,c;\lambda)\frac{t^{n}}{n!}=\frac{t^{\alpha}}{(\lambda e^{t\ln (\frac{b}{a})}+1)^{\alpha}}e^{t(x\ln c -x\ln a)}. $$
We get
\begin{eqnarray*}
\alpha\lambda\ln (\frac{b}{a})\sum_{k=0}^{n}{n \choose k}(\ln b)^{k}G_{n-k}^{(\alpha+1)}(x;a,b,c;\lambda)&=&(\alpha-n)G_{n}^{(\alpha)}(x;a,b,c;\lambda)\\
    &+& n(x\ln c -\alpha\ln a)G_{n-1}^{(\alpha)}(x;a,b,c;\lambda)
\end{eqnarray*}
\end{remark}
\begin{remark}
GI-Sang Cheon and H. M. Srivastava in \cite{Ch,Lu6} investigated the classical relationship between Bernoulli and Euler polynomials as follows
$$B_{n}(x)=\sum_{^{k=0}_{k\neq 1}}^{n}{n \choose k}B_{k}E_{n-k}(x)$$
by applying a similar Srivastava's method in \cite{Lu6} we obtain the following result for generalized Bernoulli polynomials and Genocchi numbers
\begin{eqnarray*}
B_{n}(x+y,a,b)&=&\frac{1}{2}\sum_{k=0}^{n}\frac{1}{n-k+1}{n \choose k}[B_{k}(y,a,b)+B_{k}(y+1,a,b)]G_{n-k}(x)\\
G_{n}(x+y)&=&\frac{1}{2}\sum_{k=0}^{n}{n \choose k}[G_{k}(y)+G_{k}(y+1)]E_{n-k}(x)
\end{eqnarray*}
so, because we have
\begin{equation*}
G_{n}(y+1)+G_{n}(y)=2ny^{n-1}
\end{equation*}
we obtain
\begin{equation*}
G_{n}(x+y)=\sum_{k=0}^{n}k{n \choose k}y^{k-1}E_{n-k}(x) \ \ \ \ \ \ (y\neq 0)
\end{equation*}
\end{remark}
\section{multiplication theorem for 2-variable Genocchi polynomial}
We apply the method of generating function, which are exploited
to derive further classes of partial sums involving generalized
many index many variable polynomials. In introduction we
introduced 2-variable Genocchi polynomial. An application of
2-variable Genocchi polynomials is relevant to the multiplication
theorems. In this section we develop the multiplication theorem
for 2-variable Genocchi polynomial which yields a deeper insight
into the effectiveness of this type of generalizations.
\begin{theorem}
Let $x,y\in\mathbb{R}^{+}$ and $m$ be odd, we obtain
$$G_{n}(mx,py,\lambda)=m^{n-1}\sum_{k=0}^{m-1}\lambda^{k}(-1)^{k} _{H}G_{n}\Big(x+\frac{k}{m},\frac{py}{m^{2}},\lambda^{m}\Big) $$
\end{theorem}
\begin{proof}
We know
$$\sum_{n=0}^{\infty} G_{n}(mx,py,\lambda)\frac{t^{n}}{n!}=\frac{2t e^{mxt+pyt^{2}}}{\lambda e^{t}+1} $$
and handing the R. H. S of the above equations, we defined
$$\sum_{n=0}^{\infty} G_{n}(mx,py,\lambda)\frac{t^{n}}{n!}=\frac{2t e^{mxt}}{\lambda^{m} e^{mt}+1} \frac{\lambda^{m} e^{mt}+1}{\lambda e^{t}+1} e^{pyt^{2}} .$$
By noting that
\begin{equation*}
\frac{2t
e^{mxt}}{\lambda^{m}e^{mt}+1}\frac{\lambda^{m}e^{mt}+1}{\lambda
e^{t}+1}e^{pyt^{2}}=\sum_{k=0}^{m-1}\frac{1}{m}(-1)^{k}\lambda^{k}\sum_{q=0}^{\infty}\frac{t^{q}m^{q}}{q!}G_{q}\Big(x+\frac{k}{m},\lambda^{m}\Big)\sum_{r=0}^{\infty}
\frac{t^{2r}p^{r}}{r!}y^{r}.
\end{equation*}
We get
\begin{equation*}
\sum_{n=0}^{\infty}G_{n}(mx,py,\lambda)\frac{t^{n}}{n!}=\sum_{n=0}^{\infty}t^{r}m^{n-1}\sum_{k=0}^{m-1}(-1)^{k}\lambda^{k}\sum_{r=0}^{[\frac{n}{2}]}
\frac{G_{n-2r}(x+\frac{k}{m},\lambda^{m})}{(n-2r)!r!}\Big(\frac{py}{m^{2}}\Big)^{r}
\end{equation*}
Also, by simple computation we realize that
$$_{H}G_{n}(x,y,\lambda)=\sum_{s=0}^{[\frac{n}{2}]}\frac{y^{s}G_{n-2s}(x,\lambda)}{s!(n-2s)!}  $$
So, we obtain
$$G_{n}(mx,py,\lambda)=m^{n-1}\sum_{k=0}^{m-1}(-1)^{k}\lambda^{k}_{H}G_{n}\Big(x+\frac{k}{m},\frac{py}{m^{2}},\lambda^{m}\Big) $$
Therefore proof is complete.
\end{proof}
Also, by a similar method, we get the following remark.
\begin{remark}
Let $m$ be odd and $x,y\in\mathbb{R}^{+}$, we get
$$_{H}G_{n}(mx,m^{2}y,\lambda)=m^{n-1}\sum_{\ell =0}^{m-1}(-1)^{\ell}\lambda^{\ell}_{H}G_{n}\Big(x+\frac{\ell}{m},y,\lambda^{m}\Big). $$
\end{remark}

\textbf{Acknowledgments}: The authors wishes to express his
sincere gratitude to the referee for his/her valuable suggestions
and comments.

\end{document}